\documentclass[11pt]{amsart}

\usepackage{amsthm,amssymb,amsmath,amsfonts}
\usepackage[utf8]{inputenc}
\usepackage[greek,english]{babel}
\usepackage{graphicx}
\usepackage[]{datetime}
\usepackage{aliascnt} 
\usepackage{lmodern}
\usepackage{enumitem}
\usepackage{tikz}
\usepackage{tikz-cd}
\usepackage{multicol}
\usepackage{subcaption}
\usepackage{caption}
\usepackage{alphabeta}
\usepackage{yfonts}
\usepackage{float} 

\newcommand{\mynewtheorem}[2]{
 \newaliascnt{#1}{dummy}
 \newtheorem{#1}[#1]{#2}
 \aliascntresetthe{#1}
 \expandafter\def\csname #1autorefname\endcsname{#2}
}

\theoremstyle{plain}
\mynewtheorem{theorem}{Theorem}
\mynewtheorem{proposition}{Proposition}
\mynewtheorem{corollary}{Corollary}
\mynewtheorem{lemma}{Lemma}
\mynewtheorem{exercise}{Exercise}

\theoremstyle{definition}
\mynewtheorem{definition}{Definition}
\theoremstyle{remark}
\mynewtheorem{sublemma}{Sublemma}
\mynewtheorem{claim}{Claim}
\mynewtheorem{fact}{Fact}
\mynewtheorem{remark}{Remark}
\mynewtheorem{application}{Application}
\mynewtheorem{conjecture}{Conjecture}
\mynewtheorem{question}{Question}       
\mynewtheorem{observation}{Observation}
\mynewtheorem{problem}{Problem}
\mynewtheorem{note}{Note}
\mynewtheorem{examples}{Examples}
\mynewtheorem{example}{Example}
\mynewtheorem{exercises}{Exercises}

\usepackage{hyperref,todonotes}
\usepackage{cleveref}
\usepackage{wrapfig,hyperref,cite}

\usetikzlibrary{matrix, calc, arrows}

\hypersetup{
  colorlinks   = true, 
  urlcolor     = blue, 
  linkcolor    = blue, 
  citecolor   = red 
}

\newcommand{\Z}{\mathbb{Z}}

\newcommand{\cM}{\mathcal{M}}

\title{Relations for quadratic Hodge integrals via stable maps}
\author{Georgios Politopoulos}
\address{Laboratoire AGM, 2 avenue Adolphe Chauvin, 95300, Cergy-Pontoise, France}
\address{Mathematical Institute, Leiden University, PO Box 9512, 2300 RA Leiden, The Netherlands}
\email{g.politopoulos@math.leidenuniv.nl}

\date{\today}

\begin{document}

\maketitle

\begin{abstract} Following Faber-Pandharipande, we use the virtual localization formula for the moduli space of stable maps to $\mathbb{P}^1$ to compute relations between Hodge integrals. We prove that certain generating series of these integrals are polynomials.
\end{abstract}

Let $\overline{\cM}_{g,n}$ be the moduli space of $n$-pointed genus $g$ stable curves. It is a proper
smooth  Deligne Mumford (DM) stack of dimension $3g-3+n$. We denote by $\pi:\overline{\mathcal{C}}_{g,n}\to \overline{\cM}_{g,n}$ the universal curve and by $\sigma_i: \overline{\cM}_{g,n}\to \overline{\mathcal{C}}_{g,n}$ the sections associated to the marking $i$ for all $1\leq i\leq n$. We denote by $\omega_{\overline{\mathcal{C}}_{g,n}/ \overline{\cM}_{g,n}}$ the relative dualizing sheaf of $\pi$. We will consider the following classes in $A^*(\overline{\cM}_{g,n})$:
\begin{itemize}
    \item For all $0\leq i\leq g$, $\lambda_i$ stands for the $i$-th Chern class of the Hodge bundle, i.e. the vector bundle $\mathbb{E}=\pi_*\omega_{\overline{\mathcal{C}}_{g,n}/ \overline{\cM}_{g,n}}$. For all $\alpha\in \mathbb{C}$, we denote $\Lambda_g(\alpha)=\sum_{j=0}^g\alpha^{g-j} λ_j$, and  $\Lambda^\vee_g(\alpha)=(-1)^g\Lambda_g(-\alpha)$. \footnote{Here we use the convention of ~\cite{FabPan} for $Λ^\vee_{g}(α)$ and  $Λ_{g}(α)$}
    \item For all $1\leq i\leq n$, we denote $\psi_i$ the Chern class of the cotangent line at the $i$th marking $\mathcal{L}_i=\sigma_i^*(\omega_{\overline{\mathcal{C}}_{g,n}/ \overline{\cM}_{g,n}})$. 
\end{itemize}
A {\em Hodge integral} is an intersection number of the form:
$$
\int_{\overline{\cM}_{g,n}} \psi_1^{k_1}\ldots \psi_n^{k_n}\Lambda_g(t_1)\ldots \Lambda_g(t_m),
$$
where $k_1,\ldots,k_n$ are non-negative integers and $t_1,\ldots,t_m$ are complex numbers. If $m=1,2,$ or 3, then the above integral is called a linear, double, or triple Hodge integrals respectively. Relations between linear Hodge integrals where proved in~\cite{FabPan} using the Gromov-Witten theory of $\mathbb{P}^1$ and the localization formula of~\cite{GraPan1}. This approach was also used in~\cite{FabPan1} and~\cite{TiaZho} to prove certain properties of triple Hodge integrals. Linear and triple Hodge integrals naturally appeared in the GW-theory of Calabi-Yau 3-folds, thus explaining a more abundant literature on the topic. However, double Hodge integrals have appeared recently in the Quantization of Witten-Kontsevich generating series (see~\cite{Blo}),  in the theory of spin Hurwitz numbers (see~\cite{GiaKraLew}), and in the GW theory of blow-ups of smooth surfaces (see ~\cite{Giacchetto2022TheSG}). 

In the present note, we consider the following power series in $\mathbb{C}[\alpha][\![t]\!]$ defined using double Hodge integrals
\begin{equation}\label{pol}
    P_a(α,t)=\sum_{g\geq0} t^{g}\left( \int_{\overline{\cM}_{g,n+1}}\!\!\!\!\!\!\!\!\!\!\!\!
    \frac{Λ_{g}^{\vee}(1)Λ_{g}^{\vee}(α)}{1-ψ_0}\prod_{i=1}^{n} {(2a_i+1)!!(-4ψ_i)^{a_i}}{}\right)
   \text{exp}\left(\frac{t}{24}\right)
\end{equation}
where $a=(a_1,\ldots,a_n)$ is a vector of non-negative integers. If $n=1$, we use the convention: $\int_{\overline{\cM}_{0,2}} \psi_1^a \frac{Λ_{g}^{\vee}(1)Λ_{g}^{\vee}(α))}{1- \psi_2}=(-1)^a$.

\begin{theorem}\label{th:main}
$P_a(α,t)$ is a monic polynomial in $\mathbb{C}[\alpha][t]$ of degree $|a|$ in $t$.
\end{theorem}
Here we provide the first values of $P_a(-α-1,t)$. In the list below we omit the variables $-α-1$ and $t$ in the notation:
\begin{eqnarray*}
&&P_{()}= 1\\
\\
&&P_{(1)}= t+12\\
\\
&&P_{(2)}=t^2 - 10 \alpha t + 240  \\
&&P_{(1,1)}=t^2 - 12 t   \\
\\
&&P_{(3)}=t^3 + (-77/3 \alpha - 28) t^2 + 280 t +6720    \\
&&P_{(2,1)}= t^3 + (-10\alpha - 48)t^2 +(240\alpha + 240)t  \\
&&P_{(1,1,1)}=t^3 - 72t^2+432t   \\
\\
&& P_{(4)}=t^4+  (-43\alpha - 72)t^3 + (126\alpha^2 + 756\alpha + 840)t^2  + 10080t +  241920 \\
&& P_{(3,1)}= t^4+ (-77/3 \alpha - 100) t^3+ (1232 \alpha + 1624) t^2  \\
&& P_{(2,2)}=t^4+  (20\alpha + 100)t^3+(-100\alpha^2 - 1360\alpha - 1680)t^2  \\
&& P_{(2,1,1)}=t^4+ (-10\alpha - 132)t^3 + (840\alpha + 3120)t^2 +(-8640\alpha - 8640)t 
\\
&& P_{(1,1,1,1)}=t^4 - 168t^3   + 5616t^2 -20736t. \\
\end{eqnarray*}
Considering these first values, we conjecture that $P_a$ is a polynomial of total degree $|a|$ in both variables $t$ and $α$.

\section*{Acknowledgements} I am very grateful to my PhD advisor Adrien Sauvaget
for introducing me to this problem and for his guidance and comments throughout
the whole writing of this note.

\section{Preliminaries}

We denote by $\overline{\cM}_{g,n}(\mathbb{P}^1,1)$ the moduli space of stable maps of degree $1$ to
$\mathbb{P}^1$. It is a proper DM stack of virtual dimension $2g+n$. 
Here we can define in an analogous way 
the Hodge bundle $\mathbb{E}$, the cotangent line bundles $\mathcal{L}_i$ and we denote again 
$λ_i$ and $ψ_i$ the respective Chern classes. We also have
the forgetful and evaluation maps
\begin{equation*}
    π\colon \overline{\cM}_{g,n+1}(\mathbb{P}^1,1)\to \overline{\cM}_{g,n}(\mathbb{P}^1,1),\ \text{and} \ \ \ 
    ev_i\colon \overline{\cM}_{g,n+1}(\mathbb{P}^1,1)\to \mathbb{P}^1.
\end{equation*}
Throughout this
note the enumeration of markings starts from $0$. Furthermore, $\pi$ is  the morphism that forgets the marking $p_0$ and $ev_i$ is the evaluation of a stable map to the $i$-th marked point.
The vector bundle $T:=R^1π_*(ev_0^*\mathcal{O}_{\mathbb{P}^1}(-1)))$ is of rank $g$ and we denote by $y$
its top Chern class.  We will denote: 
\begin{equation*}
    \langle \prod_{i=0}^{n-1}τ_{a_i}(ω)|y\rangle_{g,1}^{\mathbb{P}^1}:= \int_{[\overline{\cM}_{g,n}(\mathbb{P}^1,1)]^{vir}}\prod_{i=0}^{n-1}ψ_i^{a_i}ev_i^*(ω)y
\end{equation*}
where $ω$ denotes the class of a point in $\mathbb{P}^1$. 
\begin{theorem}[Localization Formula,~\cite{GraPan1}, ~\cite{FabPan}]
Let $g\in\mathbb{Z}_{\geq0}$, let $a\in \mathbb{Z}^n_{\geq0}$ such that $|a|\leq g$.  
Then, for  all complex numbers ${α}$, and $t\in \mathbb{C}^*$, we have
\begin{align*}
    \langle \prod_{i=1}^nτ_{a_i}(ω)|y\rangle_{g,1}^{\mathbb{P}^1}=
    \sum_{g_1+g_2=g}
    &\int_{\overline{\cM}_{g_1,n+1}}t^{n}\prod_{i=1}^nψ_i^{a_i}
    \frac{Λ^{\vee}_{g_1}(t)Λ_{g_1}^{\vee}(αt)}{t(t-ψ_0)}\ \\
    &\times\int_{\overline{\cM}_{g_2,1}}
    \frac{Λ_{g_2}^{\vee}(-t)Λ_{g_2}^{\vee}((α+1)t)}{-t(-t-ψ_0)}
\end{align*}
\end{theorem}
\noindent 
Here we use the convention $\int_{\overline{\cM}_{0,1}}ψ_0^a=1$.

\begin{proposition}[4.1 of \cite{TiaZho}]
For all complex numbers $α$ we have 
\begin{equation*}
   F(α,t)=1+\sum_{g>0}t^{2g}\int_{\overline{\cM}_{g,1}}\frac{Λ_{g}^{\vee}(1)Λ_{g}^{\vee}(α)}{1-ψ_0}=\text{exp}
   \left(-\frac{t^2}{24}\right).
\end{equation*}
\end{proposition}

\noindent Besides, we have the String and Dilaton equation for Hodge integrals.

\begin{proposition} Let $g,n\in\Z_{\geq0}$ such that $2g-2+n>0$.
\begin{enumerate}

    \item [(i)][Dilaton equation for Hodge integrals] Let $(a_1,...,a_{n})\in\Z_{\geq0}^{n}$
    and assume that there exist $i_0$ such that $a_{i_0}=1$. Then 
    
    \begin{equation*}
        \int_{\overline{\cM}_{g,n+1}}\frac{ψ_{i_0}\prod_{i\neq i_0}ψ_i^{a_i}\prod_{j=1}^gλ_k^{b_k}}{1-ψ_0}=
        (2g-2+n)
        \int_{\overline{\cM}_{g,n}}\frac{\prod_{i=1}^{n-1}ψ_i^{a_i}\prod_{j=1}^gλ_k^{b_k}}{1-ψ_0}.
    \end{equation*}
    
     \item [(ii)][String equation for Hodge integrals] Let $(a_1,...,a_{n})\in\Z_{\geq0}^{n}$ 
    and assume that there exist $i_0$ such that $a_{i_0}=0$. Then we have
    \begin{eqnarray*}
        \int_{\overline{\cM}_{g,n+1}}\frac{\prod_{i=1}^nψ_i^{a_i}\prod_{j=1}^gλ_k^{b_k}}{1-ψ_0}&=&
        \int_{\overline{\cM}_{g,n}}\frac{\prod_{i=1}^{n-1}ψ_i^{a_i}\prod_{j=1}^gλ_k^{b_k}}{1-ψ_0}\\&&+
        \sum_{j=1}^n\int_{\overline{\cM}_{g,n}}\frac{ψ_j^{a_j-1}\prod_{i\neq j}ψ_i^{a_i}\prod_{k=1}^gλ_k^{b_k}}{1-ψ_0}.
    \end{eqnarray*}
    
\section{The calculation}

\end{enumerate}
\end{proposition}

Note that the GW-invariant  $\langle \prod_{i=1}^nτ_{a_i}
(ω)|y\rangle_{g,1}^{\mathbb{P}^1}$ is 0 unless
    $|a|=g$ for dimensional reasons. Indeed,
    $\text{dim}_{\mathbb{C}}[\overline{\cM}_{g,n}
    (\mathbb{P}^1,1)]^{\text{vir}}=2g+n$ 
    and the cycle we are integrating is in codimension $g+|a|+n$. Using 
    the above localization formula, and  
    Lemma 2.1 of ~\cite{TiaZho} the intersection number $\langle 
    \prod_{i=1}^nτ_{a_i}(ω)|y\rangle_{g,1}^{\mathbb{P}^1}$ is expressed as:

    \begin{align*}
    &\sum_{g_1+g_2=g}\int_{\overline{\cM}_{g_1,n+1}}\!\!\!\!\!\!t^{n}\prod_{i=1}^nψ_i^{a_i}
    \frac{Λ_{g_1}^{\vee}(t)Λ_{g_1}^{\vee}(αt)}{t(t-ψ_0)}\ \cdot \int_{\overline{\cM}_{g_2,1}}
    \frac{Λ^{\vee}_{g_2}(-t)Λ_{g_2}^{\vee}((α+1)t)}{-t(-t-ψ_0)}\\
    =&
    \sum_{g_1+g_2=g} t^{|a|-g_1}(-t)^{-g_2}\int_{\overline{\cM}_{g_1,n+1}}
    \prod_{i=1}^{n}ψ^{a_i}\frac{Λ_{g_1}^{\vee}(1)Λ_{g_1}^{\vee}(α)}{1-ψ_0}
    \times \int_{\overline{\cM}_{g_2,1}}\frac{Λ_{g_2}^{\vee}(1)Λ_{g_2}^{\vee}(-(α+1))}{1-ψ_0} \\
    =& \ t^{|a|-g}\sum_{g_1+g_2=g}\int_{\overline{\cM}_{g_1,n+1}}
    \prod_{i=1}^{n}ψ^{a_i}\frac{Λ_{g_1}^{\vee}(1)Λ_{g_1}^{\vee}(α)}{1-ψ_0}\cdot \int_{\overline{\cM}_{g_2,1}}ψ_0^{3g_2-2} 
    \end{align*}
\noindent In the last equation we used Proposition 1.2 in order 
to replace $\int_{\overline{\cM}_{g_2,1}}\frac{Λ_{g_2}^{\vee}
(1)Λ_{g_2}^{\vee}(-(α+1))}{1-ψ_0}$ with 
$(-1)^{g_2}\int_{\overline{\cM}_{g_2,1}}ψ_0^{3g_2-2}$.
    
We define 
\begin{equation*}
    A_{g,a}(\alpha) = \sum_{g_1+g_2=g}\int_{\overline{\cM}_{g_1,n+1}}
    \prod_{i=1}^{n}ψ^{a_i}\frac{Λ_{g_1}^{\vee}(1)Λ_{g_1}^{\vee}(α)}{1-ψ_0}\cdot \int_{\overline{\cM}_{g_2,1}}ψ_0^{3g_2-2}.
\end{equation*}
Then, we have 
   \begin{equation*}
A_{g,a}(\alpha)=  \left\{\begin{matrix}
0 &|a|<g \\ 
 \langle \prod_{i=1}^nτ_{a_i}(ω)|y\rangle_{g,1}^{\mathbb{P}^1} & |a|=g
\end{matrix}\right.    
\end{equation*}

\noindent By the definition of $Λ_g^{\vee}(t)$ we see that $Λ_g^{\vee}(1)Λ_g^{\vee}(-(α+1))$ is a 
polynomial in $α$ of degree $g$, which actually determines the degree of $A_g(α)$.

\newpage

We now present a proof for the main result.

\begin{proof}\  [of Theorem 0.1] We begin by stating the well known fact 
\begin{equation*}
    1 + \sum_{g>0}t^{g}\int_{\overline{\cM}_{g,1}}ψ_0^{3g-2} =\exp\left({\frac{t}{24}}\right)
\end{equation*}
proven in section 3.1 of \cite{FabPan}. Now, we consider the product of $\exp\left({\frac{t}{24}}\right)$
and 
\begin{equation*}
    \sum_{g\geq0} t^{g}\left( \int_{\overline{\cM}_{g,n+1}}\!\!\!\!\!\!\!\!\!\!\!\!
    \frac{Λ_{g}^{\vee}(1)Λ_{g}^{\vee}(α)}{1-ψ_0}\prod_{i=1}^{n} {(2a_i+1)!!(-4ψ_i)^{a_i}}{}\right)
\end{equation*}
to obtain a new power series whose coefficients in degree $g$ are given by
\begin{equation*}
    \sum_{g_1+g_2=g}\int_{\overline{\cM}_{g_1,n+1}}\prod_{i=1}^{n}(2a_i+1)!!(-4)^{a_i}
    \prod_{i=1}^{n}ψ^{a_i}\frac{Λ_{g_1}^{\vee}(1)Λ_{g_1}^{\vee}(α)}{1-ψ_0}\cdot \int_{\overline{\cM}_{g_2,1}}ψ_0^{3g_2-2}
\end{equation*}
This is exactly $A_{g,a}(α)\cdot\prod_{i=1}^{n}(2a_i+1)!!(-4)^{a_i}$.
Hence, we can rewrite the power series $P_a(α,t)$ in the form 
\begin{equation*}
    P_a(α,t)=\prod_{i=1}^{n}(2a_i+1)!!(-4)^{a_i}\sum_{g\geq0}t^{g}A_{g,a}(α)
\end{equation*}
As it is computed in the start of Section $2$ we have that the numbers $A_{g,a}(α)$ vanish when $g>|a|$. Hence, we get that all
coefficients of the power series $P_a(α,t)$ vanish when $g > |a|$, 
i.e. $P_a(α,t)$ is a polynomial of degree $|a|$. 
Furthermore, the top coefficient of $P_a(\alpha,t)$, i.e. the coefficient of $t^{|a|}$ is given by
\begin{equation*}
    \langle \prod_{i=1}^n (-4)^{a_i}(2a_i+1)!! τ_{a_i}(ω)|y\rangle_{|a|,1}^{\mathbb{P}^1}.
\end{equation*}
This value is computed in ~\cite{KieLi1} and is actually equal to $1$. 
In particular, the number $\prod_{i=1}^n (-4)^{a_i}(2a_i+1)!! $ is here 
to make the polynomial monic. 
\end{proof}

We now prove several other properties of the polynomials $P_a$.

\begin{proposition} The constant term $c_0$ of $P_a(α,t)$ is non zero if and only if $n=1$ where then
$c_0=(-1)^a\prod_{i=1}^n (-4)^{a_i}(2a_i+1)!!$ or if $n>1$ and $\sum_{i=1}^na_i\leq n-2$ where then 
\begin{equation*}
    c_0=\prod_{i=1}^n (-4)^{a_i}(2a_i+1)!!\frac{(n-2)!}{a_1!\cdots (n-2-\sum{a_i})!}
\end{equation*}
\end{proposition}

\begin{proof} We only compute the integrals appearing in the constant term of this polynomial since then we only have to multiply with $\prod_{i=1}^{n}(2a_i+1)!!(-4)^{a_i}$. The integral in the constant term of $P_a(α,t)$ is given by
$\int_{\overline{\cM}_{0,n+1}}\frac{\prod_{i=1}^nψ_i^{a_i}}{1-ψ_0}$. When $n=1$, using the convention
$\int_{\overline{\cM}_{0,2}}\frac{ψ_1^{a}}{1-ψ_0}=(-1)^a$ we get that
\begin{equation*}
    c_0=(-1)^a\prod_{i=1}^n (-4)^{a_i}(2a_i+1)!!.
\end{equation*}
When $n>1$, 
if $\sum_{i=1}^na_i>n-2$, then $c_0$ is zero for dimensional reasons. Otherwise, we have 
\begin{equation*}
    \int_{\overline{\cM}_{0,n+1}}\frac{\prod_{i=1}^nψ_i^{a_i}}{1-ψ_0}=\int_{\overline{\cM}_{0,n+1}}
    ψ_0^{n-2-\sum a_i}\prod_{i=1}^nψ_i^{a_i}= \frac{(n-2)!}{a_1!\cdots (n-2-\sum{a_i})!}
\end{equation*}
and so we obtain the desired result.

\end{proof}

\begin{proposition} Let $n\geq3$. Then we have the following rules: 
\begin{enumerate}

    \item[(i)][String equation]
    \begin{equation*}
        P_{(a_1,...,a_{n-1},0)}(α,t)= P_{(a_1,...,a_{n-1})}(α,t)- \sum_{i=1}^n (8a_i+4) P_{(a_1,..,a_i-1,...,a_{n-1})}(α,t)
    \end{equation*}
    
    \item[(ii)][Dilaton equation]
    \begin{equation*}
        P_{(a_1,...,a_{n-1},1)}(α,t)=(t-12n+24) P_{(a_1,...,a_{n-1})}(α,t) - 24 t P'_{(a_1,...,a_{n-1})}(α,t))
    \end{equation*}
    
\end{enumerate}
\end{proposition}

\begin{proof} We define the power series 
\begin{equation*}
    \widetilde{P}_a(α,t)= \sum_{g\geq0}
t^{g}\left(\int_{\overline{\cM}_{g,n+1}}\prod_{i=1}^{n}ψ^{a_i}
\frac{Λ_{g}^{\vee}(1)Λ_{g}^{\vee}(α)}{1-ψ_0}\right) 
\end{equation*}
Note that the following equation holds.
\begin{equation*}
    P_a(α,t)=\prod_{i=1}^{n}(2a_i+1)!!(-4)^{a_i}\widetilde{P}_a(α,t)\exp\left(\frac{t}{24}\right)
\end{equation*}
We can rewrite the coefficients of $\widetilde{P}_a(α,t)$ as
\begin{equation*}
    \sum_{k=0}^g\sum_{j=0}^g(-1)^{g+k}(a+1)^{g-j}\int_{\overline{\cM}_{g,n+1}}
    \frac{\prod_{i=1}^nψ_i^{a_i}λ_kλ_j}{1-ψ_0}
\end{equation*}
\begin{enumerate}
    \item [(i)] Applying the String equation for Hodge integrals we obtain the following formula
    \begin{equation*}
        \widetilde{P}_{(a_1,...,a_{n-1},0)}(α,t)=\widetilde{P}_{(a_1,...,a_{n-1})}(α,t)+
        \sum_{i=1}^n\widetilde{P}_{(a_1,..,a_i-1,...,a_{n-1})}(α,t)
    \end{equation*}
    Hence, multiplying with $\prod_{i=1}^{n-1}(2a_i+1)!!(-4)^{a_i}$ $\text{exp}\left(\frac{t}{24}\right)$ we 
    obtain the desired result after a straightforward calculation.
    
    \item [(ii)] Applying Dilaton equation for Hodge integrals we obtain the following formula
    \begin{eqnarray*}
         \widetilde{P}_{(a_1,...,a_{n-1},1)}(α,t)&=&
         2\sum_{g\geq0}gt^g\int_{\overline{\cM}_{g,n-1}}\prod_{i=1}^{n-1}ψ_i^{a_i}\frac{Λ^{\vee}_g(1)Λ^{\vee}_g(α)}{1-ψ_0}
\\         &&+ (n-2)\widetilde{P}_{(a_1,...,a_{n-1})}(α,t)
    \end{eqnarray*}
Note that the the first term of the sum is equal 
to $2t\widetilde{P}'_{(a_1,...,a_{n-1})}(α,t)$. Now, 
multiplying both sides of the equation above with
\begin{equation*}
    \prod_{i=1}^{n-1}(2a_i+1)(-4)^{a_i}\text{exp}\left(\frac{t}{24}\right)
\end{equation*}
we have
    \begin{eqnarray*}
        \frac{-1}{12}P_{(a_1,...,a_{n-1},1)}(α,t)\!\!&=& \!\! (n-2)P_{(a_1,...,a_{n-1})}(α,t)\\
        && \!\!\!\!\!\!\!\! +
        2t \left(\prod_{i=1}^{n-1} (-4)^{a_i}(2a_i+1)!!\right) \widetilde{P}'_{(a_1,...,a_{n-1})}(α,t){\rm e}^{{t}/{24}}\\
         \!\! &=& \!\! (n-2)P_{(a_1,...,a_{n-1})}(α,t)\\
         && \!\!\!\!\!\!\!\! + 2t( P'_{(a_1,...,a_{n-1})}(α,t)-\frac{1}{24} P_{(a_1,...,a_{n-1})}(α,t)).
    \end{eqnarray*}
    \noindent Finally clearing the denominators we obtain the desired 
    result.
\end{enumerate}

\end{proof}

We recall Mumford's relation $\Lambda^\vee_g(1)\cdot \Lambda^\vee_g(-1)=1$
(see~\cite{Mum}). In particular,  $P_a(-1,t)$ is defined by integrals of $\psi$-classes. 

\begin{corollary} For any vector $a\in\Z^n_{\geq0}$, the power series 
\begin{equation*}
    P_a(-1,t)=\prod_{i=1}^{n}(2a_i+1)!!(-4)^{a_i}\text{exp}\left(\frac{t}{24}\right)\cdot\sum_{g\geq0}(-t)^g\int_{\overline{\cM}_{g,n+1}}
    \frac{\prod_{i=1}^nψ_i^{a_i}}{1-ψ_0}
\end{equation*}
is a polynomial of degree $|a|$. 
\end{corollary}

In this case the polynomiality as well as a closed expression were proved in~\cite{LiuXu}.

\bibliographystyle{halpha.bst}
\bibliography{biblio}

\end{document}